\documentclass[a4paper]{amsart}
\usepackage{amssymb,amsmath,float,amsfonts,latexsym,amsthm,comment,pinlabel,comment}
\usepackage{tikz}
\usetikzlibrary{arrows}
 \usepackage{tikz-cd}
\usetikzlibrary{matrix,arrows,decorations.pathmorphing}
\usetikzlibrary{decorations.markings}

\numberwithin{equation}{section}

\usepackage[mathscr]{eucal}

\usepackage{float}
\usepackage[T1]{fontenc}
\usepackage{txfonts}
\usepackage[pdftex,bookmarks=true]{hyperref}
\usepackage{graphicx}
\usepackage{subfigure}
\usepackage{wrapfig}
\usepackage{import}
\usepackage{epsfig}
\usepackage{enumerate}
\usepackage[all]{xy}
\theoremstyle{plain}
\newtheorem{thm}[subsection]{Theorem}

\newtheorem{lemma}[subsection]{Lemma}
\newtheorem{cor}[subsection]{Corollary}
\newtheorem{prop}[subsection]{Proposition}
\theoremstyle{definition}

\numberwithin{equation}{section}

\author{Subash Chandra Behera}
\address{
School of Mathematics \& Computer Science\\
Indian Institute of Technology Goa\\
At Goa College of Engineering Campus\\
Farmagudi, Ponda-403401 \\
Goa, India}
\email{subash20232102@iitgoa.ac.in}
\author{Shiv Parsad}
\address{
School of Mathematics \& Computer Science\\
Indian Institute of Technology Goa\\
At Goa College of Engineering Campus\\
Farmagudi, Ponda-403401 \\
Goa, India} 
\email{shiv@iitgoa.ac.in}
\begin{document}

\title{Some analogues of isoperimetric inequality}



\begin{abstract}  
The discrete isoperimetric inequality states that among all \( n \)-gons with a fixed area, the regular \( n \)-gon has the least perimeter. We prove analogues of the discrete isoperimetric inequality (involving circumradius or inradius) for cyclic and tangential polygons in hyperbolic geometry, considering both single and multiple polygons. Furthermore, we establish two versions of the isoperimetric inequality for multiple polygons in hyperbolic geometry with some restriction on their area or perimeter.

\end{abstract}  

\maketitle

%%%%%%%%%%%%%% Section 1 (Introduction) %%%%%%%%%%%%

\section{Introduction}

The discrete isoperimetric inequality states that among all \( n \)-gons with a fixed area, the regular \( n \)-gon has the least perimeter. This result holds not only in Euclidean geometry but also in spherical and hyperbolic geometries, with the spherical case established by László Fejes Tóth \cite{LF} and the hyperbolic case proven by Károly Bezdek \cite{BK}. For other related works see \cite{BC, BD}. 

A polygon is called \textit{tangential} if all its sides are tangent to the same circle (its incircle). A polygon is called \textit{cyclic} if all its vertices lie on the same circle (its circumcircle). Cyclic polygons have been studied by several authors \cite{JD, RG}. A cylic polygon $P$ is called centered if its circumcircle has its center in the interior of $P$. In this article, all cyclic polygons are assumed to be centered. Throughout the article, for a cyclic (resp. tangential) polygon $P$, $R(P)$ (resp. $r(P))$ denotes its circumradius (resp. inradius), and it is assumed that all the vertices of the polygon lie in the hyperbolic plane. Motivated by the classical isoperimetric inequality, we explore analogous inequalities involving the inradius or circumradius of a hyperbolic polygon. We prove the following results:

\begin{thm}\label{ThmIminPerimeter}  
For any tangential hyperbolic \( n \)-gon \( P \),  
\( \textit{Peri}(P) \geq 2n \tanh^{-1} \left( \tan(\pi/n) \sinh r(P) \right) \),  
with equality if and only if \( P \) is regular.  
 
\end{thm}

\begin{thm}\label{thmIMinPerimeter}

For any cyclic hyperbolic \( n \)-gon \( P \),
\( \textit{Peri}(P) \leq 2n \sinh^{-1} \left( \sin(\pi/n) \sinh R(P)\right) \),  
with equality if and only if \( P \) is regular.  
  
\end{thm}

An immediate consequence of Theorems \ref{ThmIminPerimeter} and \ref{thmIMinPerimeter}, which establish the relationship between the inradius and circumradius, is as follows:
\begin{cor}
For any hyperbolic \( n \)-gon \( P \),
\[
r(P) \geq  \sinh^{-1} \left( \frac{\tan(\pi/n)}{\tan \left( 2n \sinh^{-1} \left( \sin(\pi/n) \sinh R(P) \right) \right) } \right).
\]
\end{cor}
\begin{thm} \label{ThmIMaxArea}
For any tangential hyperbolic \( n \)-gon \( P \), 
\( \textit{Area}(P) \geq (n-2)\pi - 2n \cos^{-1}\left( \sin (\pi/n)\cosh r(P) \right) \),  
with equality if and only if \( P \) is regular.   
\end{thm}

\begin{thm} \label{ThmCMinArea}
For any cyclic hyperbolic \( n \)-gon \( P \),
 \( \textit{Area}(P) \geq (n-2)\pi -2n \cot^{-1} \left( \tan(\pi/n)\cosh R(P) \right) \),  
with equality if and only if \( P \) is regular.  
    
\end{thm}

Next, we prove isoperimetric type inequalities involving circumradius or inradius for multiple polygons. In particular, we prove the following results:

\begin{thm}\label{ThmTCMinPerimeter}
Let \(P_1,\dots, P_k\) be cyclic regular hyperbolic \(n\)-gons with a given total circumradius 
\(\sum_{i=1}^k R(P_i) = T\). 
Then 
\(\sum_{i=1}^k \textit{Peri}(P_i) \geq 2nk \sinh^{-1}\left(\sin(\pi/n)\sinh(T/k)\right)\),
with equality if and only if all \(P_i\) are isometric to a regular \(n\)-gon of circumradius $T/k$.
   
\end{thm}

\begin{thm}\label{ThmTIMinTperimeter}
Let \(P_1,\dots, P_k\) be tangential hyperbolic \(n\)-gons with a given total inradius 
\(\sum_{i=1}^k r(P_i) = T\).
Then 
\(\sum_{i=1}^k \textit{Peri}(P_i) \geq 2nk \tanh^{-1}\left(\tan(\pi/n)\sinh(T/k)\right)\),
with equality if and only if all \(P_i\) are isometric to a regular \(n\)-gon of inradius $T/k$.
  
\end{thm}

\begin{thm}\label{ThmTCMaxArea}

Let \(P_1,\dots, P_k\) be cyclic hyperbolic \(n\)-gons with a given total circumradius 
\(\sum_{i=1}^k R(P_i) = T\). 
Then
\(\sum_{i=1}^k Area(P_i) \geq k(n-2)\pi - 2nk \cot^{-1}\left(\tan(\pi/n)\cosh(T/k)\right)\), with equality if and only if all \(P_i\) are isometric to a regular \(n\)-gon of circumradius $T/k$.
  
\end{thm}

\begin{thm}\label{TIMinTarea}
Let \(P_1,\dots, P_k\) be tangential hyperbolic \(n\)-gons with a given total inradius 
\(\sum_{i=1}^k r(P_i) = T\).
Then
\(\sum_{i=1}^k Area(P_i) \geq k(n-2)\pi - 2nk \cos^{-1}\left( \sin(\pi/n)\cosh(T/k)\right)\)
holds, with equality if and only if all \(P_i\) are isometric to a regular \(n\)-gon of inradius $T/k$.

\end{thm}

Motivated by the work of Sanki and Vadnere \cite{BSV}, we prove isoperimetric inequalities for multiple polygons with fixed total area or fixed total perimeter with some constraints. In particular, we prove the following results:

\begin{thm}\label{Total perimeter minimize}
Let \(P_1, \dots, P_k\) be hyperbolic \(n\)-gons with a fixed total area,  
\(\sum_{i=1}^{k} \textit{Area}(P_i) = T\),  
satisfying  
\(\textit{Area}(P_i) > (n-2)\pi - 2n \sin^{-1} \left( \sqrt{1 - \sin (\pi/n)} \right)\)  
for \(i = 1, \dots, k\).  
Then, we have  \[
\sum_{i=1}^k \textit{Peri}(P_i) \geq 2nk \cosh^{-1} \left(\frac{\cos (\pi/n)}{\sin\left((n-2)\pi - T/k\right)/2n} \right),
\]
with equality if and only if all \(P_i\) are isometric to a regular polygon of area \(T/k\).
\end{thm}

\begin{thm}\label{minimize total Area main theorem}
Let \(P_1,\dots, P_k\) be hyperbolic \(n\)-gons with a fixed total perimeter 
\(\sum_{i=1}^k Peri(P_i) = T\) satisfying $Peri(P_i)>2n\cosh^{-1}\sqrt{1+\sin(\pi/n)}$ for $i=1,\dots, k$. Then, we have
\[\sum_{i=1}^k \textit{Area}(P_i) \leq k(n-2)\pi - 2nk \sin^{-1}\left(\frac{\cos(\pi/n)}{\cos(T/(2nk))}\right),\]  
with equality if and only if all \(P_i\) isometric to a regular polygon of perimeter $T/k$.

\end{thm}

The motivation to prove Theorem~\ref{Total perimeter minimize} was to find the minimum length of uniform filling systems \cite{Uniformfilling}. The main idea is to convert each problem into an optimization problem with an objective function with a constraint. In order to solve the optimization problem, we make use of Lemma \ref{convex lemma} and hyperbolic trigonometry. In these kind of problems, showing the existence of optima is a difficult task. A unique feature of the use of Lemma \ref{convex lemma} is that it guarantees the existence and uniqueness of optima whenever it is applicable.
\section{Preliminaries}
In this section, we present hyperbolic trigonometry formulas for a right-angled hyperbolic triangle and include expressions for area and perimeter, which are essential for proving our main result in Section \ref{Proof of main results}.

\begin{lemma}\label{sine and cosine rules} \cite{PB}
Let \( ABC \) be a hyperbolic triangle with side lengths \( a, b, c \), where the side of length \( a \) is opposite to angle \( A \) and there is a right angle at \( A \) (see Figure \ref{RightAngledTriangle}). Then, the following relations hold:
\begin{figure}[H]
    \centering
    \includegraphics[width=0.25\linewidth]{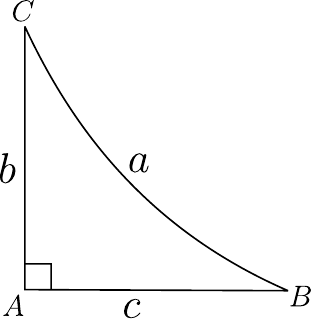}
    \caption{}
    \label{RightAngledTriangle}
\end{figure}
\vspace{2mm}
\noindent%
\begin{minipage}{0.7\textwidth}
\begin{enumerate}
\item[(i)]  \(\cosh a=\cosh b \cosh c \),  
\item[(ii)]  \(\cosh a= \cot B \cot C\),
\item[(iii)]  \(\sinh b= \sin B \sinh a\),  
\item[(iv)]  \(\sinh c=\cot B \tanh b\),  
\item[(v)]  \(\cos C =\cosh c \sin B\),  
\item[(vi)]  \(\cos B =\tanh c \coth a\).  
\end{enumerate}
\end{minipage}%
\begin{minipage}{0.25\textwidth}
\hfill%

\end{minipage}
\end{lemma}

\begin{prop}
    Let $P$ be a regular hyperbolic $n$-gon with interior angle $\theta$. The perimeter of $P$ is given by
    \(
    \text{Peri}(P) = 2n \cosh^{-1}\left( \frac{\cos(\pi/n)}{\sin\left( \theta/2 \right)} \right).
    \)
\end{prop}
\begin{proof}

\begin{figure}[htbp]
    \centering
    \includegraphics[width=0.3\linewidth]{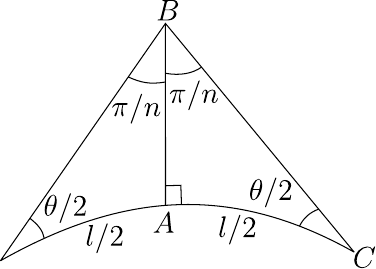}
    \caption{A triangular section of regular hyperbolic $n$-gon.}
    \label{Perimeter}
\end{figure}

Let \( B \) be the circumcenter of the polygon \( P \) and the length of each side of \( P \) is \( 2\ell \). The perpendicular projection of \( B \) onto any side bisects both the side and the angle at \( B \) (see Figure \ref{Perimeter}).

 By Lemma \ref{sine and cosine rules}, we have
   \(
   \frac{\ell}{2} = \cosh^{-1}\left( \frac{\cos(\pi/n)}{\sin(\theta/2)} \right).
   \)
   Therefore, the perimeter of $P$ is
   \(
   \text{Peri}(P) = n\ell = 2n \cosh^{-1}\left( \frac{\cos(\pi/n)}{\sin (\theta/2)} \right).
   \)
\end{proof}
\begin{thm}[Gauss- Bonnet]
The area of a hyperbolic \( n \)-gon \( P \) with interior angles \( \theta_1, \dots, \theta_n \) is given by the formula:  
\(
\text{Area}(P) = (n - 2)\pi - (\theta_1 + \dots + \theta_n)
\).
\end{thm}
\begin{proof}
    See \cite{AFB}
\end{proof}

\section{A convexity lemma}
The following lemma is repeatedly used throughout this article to prove various optimization results.

\begin{lemma}\label{convex lemma}
    Let \( f \) be a convex and twice differentiable function defined on an open interval \( I \). Define the function \( F: I^k \to \mathbb{R} \) as
    \(
    F(x_1, x_2, \dots, x_k) = f(x_1) + f(x_2) + \dots + f(x_k),
    \)
    where \( x_1, x_2, \dots, x_k \in I \). Consider the following optimization problem:
    \[
    \text{Minimize} \quad F(x_1, x_2, \dots, x_k)
    \quad \text{subject to the constraint} \quad
    x_1 + x_2 + \dots + x_k = c,
    \]
    where \( c \in \mathbb{R} \) is a constant such that \(\frac{c}{k}\in I\). The global minimum is attained at the point
    \(
    \left( \frac{c}{k}, \dots, \frac{c}{k} \right).
    \)
\end{lemma}

\begin{proof}
Since \( f \) is convex, we have \( f''(x) \geq 0 \) for all \( x \in I \).

Let $(a_1,\dots,a_k)\in I^k$ and let $a_i=\frac{c}{k}+h_i$ for some $h_i$ for $1\leq i\leq k$.
Using Taylor's theorem, for each \( i \), there exists \( c_i \) such that
\[
F\left(a_1,\dots, a_k\right)=F\left (\frac{c}{k}+ h_1, \dots,\frac{c}{k} + h_k\right) = \sum_{i=1}^{k} f\left(\frac{c}{k} + h_i\right) = \sum_{i=1}^{k} \left[ f\left(\frac{c}{k}\right) + h_i f'\left(\frac{c}{k}\right) + \frac{h_i^2 f''(c_i)}{2} \right].
\]
Since \( f''(x) \geq 0 \), the quadratic term is non-negative, so
\[
F\left(\frac{c}{k} + h_1, \dots, \frac{c}{k} + h_k\right) \geq \sum_{i=1}^{k} \left[ f\left(\frac{c}{k}\right) + h_i f'\left(\frac{c}{k}\right) \right].
\]

Given the constraint \( a_1 + \dots + a_k = c \), we have
\[
\sum_{i=1}^{k} \left(\frac{c}{k} + h_i\right) = c.
\]
This simplifies to
\[
c + \sum_{i=1}^{k} h_i = c \implies \sum_{i=1}^{k} h_i = 0.
\]
Thus,
\[
F\left(a_1,\dots, a_k\right) \geq \sum_{i=1}^{k} f\left(\frac{c}{k}\right) = F\left(\frac{c}{k}, \dots, \frac{c}{k}\right).
\]
Therefore, \( F \) attains a global minimum at \( \left(\frac{c}{k}, \dots, \frac{c}{k}\right) \).

\end{proof}

\begin{cor}\label{Concave corolary}
    If $f$ is concave, then $F$ has global maximum at \(\left(\frac{c}{k}, \dots, \frac{c}{k}\right) \). 
\end{cor}
\section{Proof of main results}\label{Proof of main results}

In this section, we prove Theorem \ref{ThmIminPerimeter} -- \ref{minimize total Area main theorem}. 
We describe the Figure \ref{Figure 1 & 2}, which we use repeatedly in our proofs. Consider an \( n \)-sided tangential polygon \( P \) with inradius \( c \) (see Figure \ref{Figure-1}). Let \( \theta \) be the angle at the incenter \( B \), formed by the radii drawn to two consecutive points of tangency on the incircle. Let \( b \) be the length of the tangent segments from these points to the vertices where adjacent tangents meet. Let $\phi$ be the interior angle of $P$ at the vertex.  Note that the line segment from the incenter to a vertex, where the two tangents intersect, bisects the angle \( \theta \) and \(\phi\).

Now, consider an \( n \)-sided cyclic hyperbolic polygon with circumradius \( a \) (see Figure \ref{Figure-2}). Let \( \theta \) be the angle at the circumcenter \( B \), formed by the radii drawn to the endpoints of a side of length \( 2c \). The line from \( B \) represents the perpendicular projection onto the corresponding side of the polygon, bisecting both the angle \( \theta \) and the side. Let $\phi$ be the angle between the side of the polygon and the inradius.

\begin{figure}[H]
    \centering
    \subfigure[Tangential polygon]{
        \includegraphics[width=0.4\textwidth]{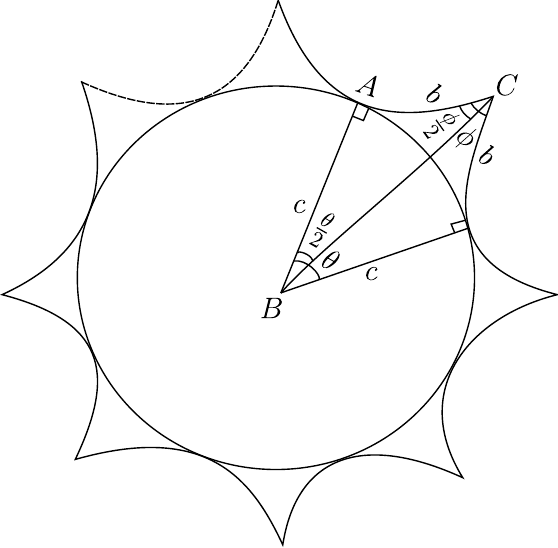}
        \label{Figure-1}
    }
    \hfill
    \subfigure[Cyclic polygon]{
        \includegraphics[width=0.35\textwidth]{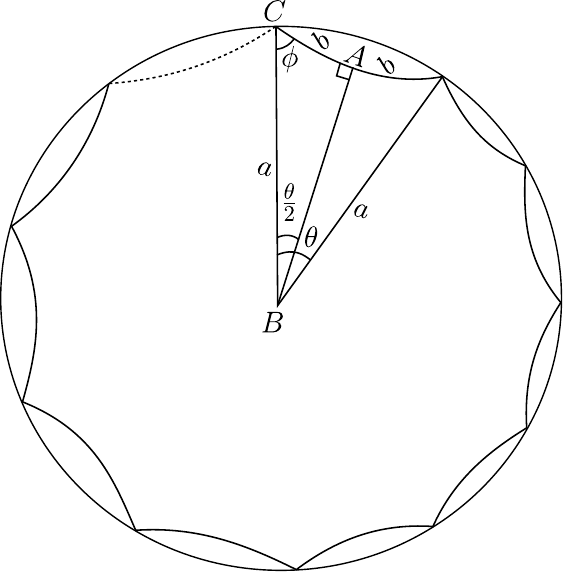}
        \label{Figure-2}
    }
    \caption{Tangential and cyclic polygons}
    \label{Figure 1 & 2}
\end{figure}

\subsection{Isoperimetric type inequalities for a cyclic or tangential polygon }
In this subsection, we prove Theorem \ref{ThmIminPerimeter}-- \ref{ThmCMinArea}

\begin{proof}[Proof of Theorem \ref{ThmIminPerimeter}]  

Let \( P \) be the polygon as described in Figure~\ref{Figure-1}, with \( \theta = \theta_i \) and \( c = r \).

By Lemma \ref{sine and cosine rules}, we have  
\[
\tan (\theta_i/2) = \frac{\tanh b}{\sinh r}.
\]  
\[\implies
b(\theta_i) = \tanh^{-1}(\sinh r \tan (\theta_i/2)).
\]
Differentiating,  
\[
b'(\theta_i) = \frac{\sinh r}{2} \cdot \frac{\sec^2(\theta_i/2)}{1 - \sinh^2 r \tan^2(\theta_i/2)}.
\]
  
\[\implies
b''(\theta_i) = \frac{\sinh r}{2} \cdot \frac{\sec^2(\theta_i/2) \tan(\theta_i/2) (1 + \sinh^2 r)}{(1 - \sinh^2 r \tan^2(\theta_i/2))^2} > 0, \quad \text{for } 0<\theta_i<\pi.
\]
Thus, \( b \) is a convex function of \( \theta_i \). By Lemma \ref{convex lemma}, minimizing the perimeter  
\[
\sum_{i=1}^{n} 2nb(\theta_i)
\]
under the constraint  
\[
\sum_{i=1}^{n} \theta_i = 2\pi
\]
is achieved when all angles \( \theta_i \) are equal, that is, \( \theta_i = 2\pi/n \). This implies all the sides and angles of the polygon are equal, thus it's regular.
Thus,
\[
\operatorname{Peri}(P) \geq 2n \tanh^{-1} \left( \tan(\pi/n) \sinh r \right).
\]

\end{proof}

\begin{proof}[Proof of Theorem \ref{thmIMinPerimeter}]
     
Let \( P \) be the polygon as described in Figure~\ref{Figure-2}, with \( \theta = \theta_i \) and \( a =R  \).

    By Lemma \ref{sine and cosine rules}, we have:
    \[
    \sin(\theta_i/2) = \frac{\sinh b}{\sinh R}.
    \]
   
    \[
    \implies
   b(\theta_i) = \sinh^{-1}\left(\sin(\theta_i/2) \sinh R\right).
    \]
 \[\implies
 b'(\theta)=\frac{\sinh R}{2}\frac{\cos(\theta_i/2)}{\sqrt{1+\sinh^{2}R}\sin^2(\theta_i/2)}
 \]
\[
 \implies b''(\theta_i)=-\frac{\sinh R}{4} \frac{\sin(\theta_i/2)((1+\sinh^2R \sin^2(\theta_i/2))^{\frac{3}{2}})+\cos^2(\theta/2)\sin(\theta_i/2)}{(1+\sinh^2R \sin^2(\theta_i/2))^{\frac{3}{2}}}< 0 \textit{ for all } 0<\theta_i<\pi
\]

   The side length \( b \) is a concave function of \( \theta_i \). 

    Therefore, by Corollary \ref{Concave corolary} the solution to the problem:
    \[
    \text{Maximize } \sum_{i=1}^n 2 b(\theta_i),
    \]
    under the constraint \[ \sum_{i=1}^n \theta_i = 2\pi \] is achieved when all angles \( \theta_i \) are equal, that is, \( \theta_i = 2\pi/n \) for all $i=1,\dots,n$. This implies all the sides and angles of the polygon are equal, thus it's regular.
Thus, \( \textit{Peri}(P) \leq 2n \sinh^{-1} \left( \sin(\pi/n) \sinh R \right) \)

\end{proof}

\begin{proof} [Proof of Theorem \ref{ThmIMaxArea}]

Let \( P \) be the polygon as described in Figure~\ref{Figure-1}, with \( \theta = \theta_i \) and \( c =r  \).

By Lemma \ref{sine and cosine rules}, we have  
\[
\cosh r = \frac{\cos (\phi/2)}{\sin(\theta_i/2)}.
\]

\[\implies
\phi(\theta_i) =2 \cos^{-1} \left( \sin (\theta_i/2) \cosh r \right).
\]

Taking the derivative with respect to \( \theta_i \), we get:
\[
\frac{\phi'(\theta_i)}{\cosh r} = \frac{- \cos (\theta_i/2)}{\sqrt{1 - \cosh^2 r \sin^2 (\theta_i/2)}}.
\]

Taking the second derivative with respect to \( \theta_i \), we obtain:
\[
\frac{\phi_i''(\theta_i)}{\cosh r} = - \frac{\sin (\theta_i/2)(cosh^2r-1)}{\left( 1 - \cosh^2 r \sin^2 (\theta_i/2) \right)^{3/2}} < 0 \quad \text{for all} \quad 0 < \theta_i < \pi.
\]

Since \( \phi \) is a convex function of \( \theta_i \), by Corollary \ref{Concave corolary}, the problem of maximizing
\[
\sum_{i=1}^{n} \phi(\theta_i)
\]
subject to the constraint
\[
\sum_{i=1}^{n} \theta_i = 2\pi
\]
attains its maximum when all angles \( \theta_i \) are equal, that is, \( \theta_i = 2\pi/n \) for all $i=1,\dots,n$. This condition corresponds to minimizing the area since the area of $P$ is 
\[
\text{Area(P)} = (n-2)\pi - \sum_{i=1}^{n} \phi_i.
\]
Therefore, the area is maximized when the polygon is regular. 

Since all \( \theta_i \)'s are equal, that is, $\theta_i=2\pi/n$, it follows that all sides and angles of the polygon are equal, and hence the polygon is regular. 
Thus, \( \textit{Area}(P) \geq (n-2)\pi - 2n \cos^{-1}\left( \sin (\pi/n)\cosh r \right) \).

\end{proof}

\begin{proof}[Proof of Theorem \ref{ThmCMinArea}]
   Let \( P \) be the polygon as described in Figure~\ref{Figure-2}, with \( \theta = \theta_i \) and \( a =R  \).

By Lemma \ref{sine and cosine rules}, we have:
\[
\phi(\theta_i) = \cot^{-1}(\cosh R \tan(\theta_i/2))
\]

Taking the derivative with respect to \( \theta_i \), we get:
\[
\frac{-2\phi'(\theta_i)}{\cosh R} = \frac{\sec^2(\theta_i/2)}{1 + \cosh^2 R \tan^2(\theta_i/2)}
\]

Taking the second derivative with respect to \( \theta_i \), we obtain:
\[
\frac{-2\phi'(\theta_i)}{\cosh R} = \frac{\tan(\theta_i/2) (\cosh^2 R - 1) \tan^2(\theta_i/2)}{(1 + \cosh^2 R \tan^2(\theta_i/2))^2}
\]

Since \(\cosh^2 R - 1 > 0\), we conclude \(\phi''(\theta_i) < 0\).

    This shows that \(\phi\) is a concave function of $\theta_i$. By Corollary \ref{Concave corolary}, the optimization problem:

    \[
    \text{Maximize } \sum_{i=1}^n \phi_i
    \]

    under the constraint:

    \[
    \sum_{i=1}^n \theta_i = 2\pi
    \]

    attains its maximum when all angles \( \theta_i \) are equal, that is,  \( \theta_i = 2\pi/n \) for all $i=1,\dots,n$. This implies that the polygon is regular.

    This condition corresponds to minizing the area since the area is given by
\[
\text{Area} = (n-2)\pi - \sum_{i=1}^{n} \phi_i.
\]

Thus, \( \textit{Area}(P) \geq (n-2)\pi -2n \cot^{-1} \left(\cosh R \tan(\pi/n) \right) \).
\end{proof}

\subsection{Isoperimetric type inequalties for multiple tangential or cylic polygons}
In this subsection, we prove Theorem \ref{ThmTCMinPerimeter} -- \ref{TIMinTarea}

\begin{proof}[Proof of Theorem \ref{ThmTCMinPerimeter}]

Let \( P_i \) be the polygon as described in as described in Figure~\ref{Figure-2}, with \( \theta = 2\pi/n \) and \( a =R_i  \). Our goal is to solve the following constrained minimization problem:

\[
\text{Minimize} \quad \sum_{i=1}^{k} \text{Peri}(P_i),
\]
subject to the constraint
\[
\sum_{i=1}^{k} R_i= T,
\].

 By Lemma \ref{sine and cosine rules}, we have the relation:
\[
b(R_i) = \sinh^{-1} \left( \sin (\pi/n) \sinh R_i \right).
\]

\[\implies
b'(R_i) = \frac{\sin (\pi/n) \cosh R_i}{\sqrt{\sin^2 (\pi/n) \sinh^2 R_i + 1}}.
\]
\[\implies
b''(R_i) = \frac{\sin (\pi/n) \sinh R_i \cos^2 (\pi/n)}{(\sin^2 (\pi/n) \sinh^2+1)^\frac{3}{2}}>0.
\]

By Lemma \ref{convex lemma}, the total perimeter
\[
\sum_{i=1}^{k} \text{Peri}(P_i) = \sum_{i=1}^{k} n b(R_i)
\]
attains its minimum when all \( R_i \) are equal. i.e $R_i=T/k$. Consequently, all \( P_i \) are isometric to a regular polygon of circumradius $T/k$.
Thus, \(\sum_{i=1}^k \operatorname{Peri}(P_i) \geq 2nk \sinh^{-1}\left(\sin(\pi/n)\sinh(T/k)\right)\).

\end{proof}

\begin{proof}[Proof of Theorem \ref{ThmTIMinTperimeter}]

By Theorem \ref{ThmIminPerimeter}, without loss of generality, we can assume all $P_i$ are regular.
Let the polygon \( P \) as described in Figure~\ref{Figure-1}, with \( \theta = 2\pi/n \) and \( c =r_i  \).
By Lemma \ref{sine and cosine rules}, we have:
\[
b(r_i)=\tanh^{-1}(\tan(\pi/n)\sinh r_i)
\]
 One can see that $b''(\theta_i)>0$. Similar arguments work as done in the case of Theorem \ref{ThmTCMinPerimeter}.
\end{proof}

\begin{proof}[Proof of Theorem \ref{ThmTCMaxArea}]

By Theorem \ref{ThmCMinArea}, without loss of generality, we can assume all $P_i$ are regular.
Let \( P_i \) be the polygon as described in as described in Figure~\ref{Figure-2}, with \( \theta = 2\pi/n \) and \( a =R_i  \).

By Lemma \ref{sine and cosine rules}, we have:
\[
\phi(R_i)=\cot^{-1}\left({\tan(\pi/n)\cosh R_i}\right)
\]
One can see that $\phi''(R_i)<0$. Similar arguments work as done in the case of Theorem \ref{ThmTCMinPerimeter}.
\end{proof}

\begin{proof}[Proof of Theorem \ref{TIMinTarea}]
   
Let \( P_i \) be the polygon as described in Figure~\ref{Figure-1}, with \( \theta = 2\pi/n\) and \( c = r_i \).

By Lemma \ref{sine and cosine rules}, we have:

\[
\phi(r_i)=\cos^{-1}(\sin(\pi/n)\cosh r_i)
\]
    
One can see that $\phi''(r_i)<0$.
    Similar arguments work as done in the case of Theorem \ref{ThmTCMinPerimeter}.
\end{proof}

\subsection{Isoperimetric inequality for multiple polygons}

\begin{proof}[Proof of Theorem \ref{Total perimeter minimize}]

Without loss of generality, we assume all \( P_i \) are regular \cite{BK}.

Let \( \theta_i \) denotes the interior angle of \( P_i \). The perimeter of \( P_i \) is  
\[
\textit{Peri}(P_i) = 2n \cosh^{-1} \!\Biggl(\frac{\cos(\pi/n)}{\sin(\theta_i/2)}\Biggr).
\]
The total area is $T$, which implies  
\[
\sum_{i=1}^{k} \theta_i =\frac{(n-2)k\pi - T}{n}.
\]
We aim to minimize  
\[
\sum_{i=1}^{k} 2n \cosh^{-1}\!\Biggl(\frac{\cos(\pi/n)}{\sin(\theta_i/2)}\Biggr).
\]
Let $$f(\theta) = \cosh^{-1}\left(\frac{\cos(\pi/n)}{\sin(\theta/2)}\right)$$

Then we have,  
\[
f'(\theta) = -\frac{\cos(\pi/n)}{2} \cdot \frac{\cos(\theta/2)}{\sin(\theta/2)\sqrt{\cos^2(\pi/n) - \sin^2(\theta/2)}}.
\]

We obtain  
\[
f''(\theta) = \frac{\cos(\pi/n)}{4\Bigl[\cos^2(\pi/n) - \sin^2(\theta/2)\Bigr]^{3/2}}
\Biggl[\frac{\cos^2(\pi/n)}{\sin^2(\theta/2)} - 2 + \sin^2(\theta/2)\Biggr].
\]

It follows that  
\( f''(\theta_i) > 0 \) if \( \theta_i < 2\sin^{-1}\Bigl(\sqrt{1 - \sin(\pi/n)}\Bigr) \), which implies  
\(
\text{Area}(P_i) > (n-2)\pi - 2n \sin^{-1} \sqrt{1 - \sin(\pi/n)}.
\)
By applying Lemma \ref{convex lemma}, the problem is minimized when \( \theta_i \) are equal, that is,  
$\theta_i = 2 \pi/n$ $\text{for all } i=1,\dots,k$. Therefore,  \( P_i \) are isometric to a regular polygon of area $T/k$. Thus, \(
\sum_{i=1}^k \textit{Peri}(P_i) \geq 2nk \cosh^{-1} \left[\cos (\pi/n)/\sin\left(((n-2)\pi - T/k)/(2n)\right) \right].
\)
\end{proof}

\begin{proof}[Proof of Theorem \ref{minimize total Area main theorem} ]

Without loss of generality, we can assume all $P_i$ are regular \cite{BK}.
We aim to solve the following maximization problem:

\[
\max \sum_{i=1}^{k} \text{Area}(P_i)
\]

subject to the constraint

\[
\sum_{i=1}^{k} \text{Peri}(P_i) = T
\]

The area of a hyperbolic regular \( n \)-gon \( P_i \) is given by
\(
\text{Area}(P_i) = (n - 2) \pi - n \theta_i
\), where \( \theta_i \) is the interior angle.

The perimeter \( \text{Peri}(P_i) \) is given by  

\[
\text{Peri}(P_i) = 2n \cosh^{-1} \left( \frac{\cos (\pi/n)}{\sin (\theta_i/2)} \right),
\]
Let \( \text{Peri}(P_i)=x_i \)
\[\implies
\theta_i = 2 \sin^{-1} \left( \frac{\cos (\pi/n)}{\cosh (x_i/2n)} \right),
\].  

Thus, the original problem is reduced to minimizing  

\[
\sum_{i=1}^{k} 2 \sin^{-1} \left( \frac{\cos (\pi/n)}{\cosh (x_i/2n)} \right)
\]

subject to the constraint  

\[
\sum_{i=1}^{k} x_i = T.
\]
Let \(
f(x) = \sin^{-1} \left((\cos (\pi/n))/(\cosh (x/2n)) \right).
\)
 We have $f''(x)>0$ for  $x>2n\cosh^{-1}\sqrt{1+\sin(\pi/n)}$. 

By Lemma \ref{convex lemma}, the minimum attains when \( x_i \) are equal, that is, \( x_i = T/k\) for all $i=1,\dots,n$. Thus, $\sum_{i=1}^k \textit{Area}(P_i) \leq k(n-2)\pi - 2nk \sin^{-1}\left[(\cos(\pi/n))/(\cos(T/(2nk))\right]$.
\end{proof}

 \section{Concluding remarks}  

A question to explore is whether the statements of Theorems \ref{thmIMinPerimeter}, \ref{ThmCMinArea}, \ref{ThmTCMinPerimeter}, and \ref{ThmTCMaxArea} remain valid when \( P \) is not centered. Additionally, do the conclusions of Theorems \ref{ThmIminPerimeter}, \ref{ThmIMaxArea}, \ref{ThmTIMinTperimeter}, and \ref{TIMinTarea} still hold if \( P \) is not tangential?  

\section*{Acknowledgements} 
The First author would like to thank the Council of Scientific and Industrial Research (File Number: 09/1290(0005)2020-EMR-I ) for providing financial support. The second author gratefully acknowledges the financial support from the Science and Engineering Research Board (SERB), Government of India through MATRICS grant (File Number: MTR/2021/000067).

\bibliographystyle{plain}
\bibliography{Bibilography}

\end{document}